\documentclass{proc-l}
\usepackage{amsfonts,latexsym,amsthm,amssymb,amsmath,amscd,euscript,tikz}

\usepackage{lmodern,url,enumerate,mathtools,microtype}
\usepackage[hmargin=1in,vmargin=1in]{geometry}
\usepackage{graphicx}
\usepackage{tikz}
\usepackage{color}

\usepackage{hyperref}
    \hypersetup{colorlinks=true,citecolor=blue,linkcolor = blue,urlcolor =black,linkbordercolor={1 0 0}}
\usepackage{float}

\newtheorem{theorem}{{Theorem}}[section]
\newtheorem*{theorem*}{Theorem}
\newtheorem{lemma}[theorem]{Lemma}
\newtheorem{proposition}[theorem]{Proposition}
\newtheorem{corollary}[theorem]{Corollary}
\newtheorem*{corollary*}{Corollary}

\newtheorem{definition}[theorem]{Definition}

\newcommand{\ve}{\varepsilon}

\newcommand{\mr}[1]{{\rm #1}}

\usepackage{wrapfig,caption}
\newcommand{\cB}{\mathcal{B}}

\newcommand{\cU}{\mathcal{U}}\newcommand{\cV}{\mathcal{V}}
\newcommand{\cX}{\mathcal{X}}

\newcommand{\bC}{\mathbb{C}}

\newcommand{\nc}{\newcommand}

\nc{\dbar}{\overline{\partial}}

\nc{\p}{\partial}
\nc{\ol}{\overline}
\nc{\ul}{\underline}
\nc{\iddbar}{\sqrt{-1}\partial \overline{\partial}}
\newcommand{\RNum}[1]{\uppercase\expandafter{\romannumeral #1\relax}}

\nc{\pa}{\partial}

\newif\ifdraft
\drafttrue 
\draftfalse
\nc{\edit}[1]{\ifdraft
{\textcolor{red}{#1}}\else{#1}\fi}

\nc{\Sz}{Sz\'ekelyhidi\ }

\title[CY Metrics from Complete Intersections]{Complete Calabi-Yau Metrics from Smoothing Calabi-Yau Complete Intersections} 

\date{February 19, 2024}
\author[Benjy J. Firester]{Benjy J. Firester}
\address{Department of Mathematics, Harvard University}
\email{benjaminfirester@college.harvard.edu}

\begin{document}
\maketitle

\begin{abstract}
We construct complete Calabi-Yau metrics on non-compact manifolds that are smoothings of an initial complete intersection $V_0$ that is a Calabi-Yau cone, extending the work of \Sz \cite{GaborPaper}.
The constructed Calabi-Yau manifold has tangent cone at infinity given by $\bC \times V_0$.
This construction produces Calabi-Yau metrics with fibers having varying complex structures and possibly isolated singularities.
\end{abstract}

\section{Introduction}
Non-compact Calabi-Yau manifolds play a pivotal role in K\"ahler geometry.
These manifolds have been widely examined since the seminal works of Yau \cite{Yau}, Cheng-Yau \cite{YauChenPaper} and Tian-Yau \cite{YauTian1,YauTianII}, and there are still many outstanding open questions.
Recent examples of complete Calabi-Yau metrics on $\bC^n$ with maximal volume growth have been produced \cite{ConlonRonanJ2015AcCm,conlon-rochon,LiCYC3,GaborPaper}, disproving a conjecture of Tian that the Euclidean metric was the unique such Calabi-Yau metric.
Works by Li \cite{LiCYC3}, Conlon-Rochon \cite{conlon-rochon} and \Sz \cite{GaborPaper} have constructed Calabi-Yau metrics on $\bC^n$ with tangent cones given by $V \times \bC$ for certain $(n-1)$-dimensional Calabi-Yau cones $V$ with smooth cross-section.
In particular, the work of \Sz \cite{GaborPaper} produces Calabi-Yau metrics on $\bC^n$ realized as the total space of a smoothing of an affine hypersurface admitting a Calabi-Yau cone metric.
A natural question to ask is if an $n$-dimensional complete intersection $V_0$ cut out by $k$ quasi-homogeneous polynomials can be used to construct a complete Calabi-Yau metric on $\bC^{n+k}$ with maximal volume growth and tangent cone at infinity given by $V_0 \times \bC^k$.
We present that a generic linear slice of such a smoothing carries a Calabi-Yau metric with tangent cone at infinity $V_0 \times \bC$.
\edit{More recent examples of Calabi-Yau metrics with singular tangent cones have been produced in \cite{chiu2022nonuniqueness,conlon2023warped}.
It is still open to produce examples with highly singular tangent cones of the form $\bC^k \times V_0$ with $k \geq 2$ and $V_0$ having smooth cross-section. The current method encounters difficulties in producing the initial approximately Ricci-flat metric which utilizes that the total space is a one-dimensional smoothing of the initial Calabi-Yau cone.
All such constructions} have tangent cones at infinity from Cheeger-Colding \cite{CheegerJeff1997Otso} and are expected to be unique \cite{ColdingUniqueTanCone,DonaldsonSun}.
In particular, this linear slice is no longer $\bC^*$-equivariant with respect to the action defining the quasi-homogeneous polynomials when the degrees are different.

Let $\edit{V_0 = V(f_1,\ldots,f_k) \subset \bC^{n+k}}$ be an \edit{$n$-dimensional,} quasi-homogeneous, klt complete intersection, where the weighted degree of $f_i$ is $d_i$ with respect to some weight vector $\xi = (w_1,\ldots,w_{n+k})$. We assume that $V_0$ admits a conical Calabi-Yau metric with respect to the rescaling action induced by $\xi$, or equivalently \cite{CollinsTristan2019SmaK}, $(V_0,\xi)$ is K-stable. 
We consider a generic linear smoothing $\cX$ fixing constants $p_1,\ldots, p_k$ given by
\[
\edit{\cX = \{f_1(x) - zp_1 = \cdots = f_k(x) - zp_k = 0\}\subset \bC^{n+k}_x \times \bC_z}.
\]
\edit{The space $\cX$ is the total space of a family that smooths out $V_0$ in that the total space is smooth and has asymptotically conical fibers degenerating to $V_0$}.

The main difficulty in our construction not present in \cite{GaborPaper} is that when not all $d_i$ are equal, the rescaling action does not take the fibers above ${z_0}$ to a rescaling ${\lambda z_0}$.
That is, it does not preserve $\cX$ as considered in $\bC^{n+k+1}$.
In $\cX$, the fibers have potentially different complex structures including isolated singularities.

We construct a form defined outside a sufficiently large set that restrict to a Calabi-Yau metric on each fiber.
By defining this form only near infinity, we guarantee that all the fibers are smooth.

Because each fiber carries a different complex structure, it is necessary to construct a suitable model at infinity of the fibers.
To do this, we consider the fiber over the point that is a suitably rescaled limit of $zp$ as $z \to \infty$.
This will be represented by a particular fiber that does not lie in $\cX$, unless each fiber is the same, as is the case in \cite{GaborPaper}.
We will require that this space is smooth, and when it is, we say that $\cX$ is smooth at infinity.
This definition and construction is given in full detail below in Definition \ref{def:ell&p'}.

Our main theorem constructs a Calabi-Yau metric \edit{on $\cX$, the total space smoothing $V_0$.}

\begin{theorem}\label{thm:mainThm}
If $V_0$ admits a Calabi-Yau cone metric $\omega_0$ and $\cX$ is smooth at infinity as defined below in Definition \ref{def:ell&p'}, then there exists a complete Calabi-Yau metric on $\cX$ with tangent cone at infinity given by $(\bC \times V_0, \iddbar |z|^2 + \omega_0)$.
\end{theorem}
\subsection*{Acknowledgments}
I am very thankful to Tristan C. Collins for his dedicated and excellent mentorship, support, and advice throughout this project as well as in framing research questions. 
Additional thanks to Joe Harris for his guidance and mentorship.
This work was funded in part by the Harvard College Research Program (HCRP).

\section{Setup}
\label{1Setup}
Consider a Calabi-Yau cone $V_0$ of complex dimension $n$, algebraically expressed as a complete intersection cut out by polynomials $f_i$:
\[
V_0 = \{f_1(x) = \cdots = f_k(x) = 0 \}\subset \bC^{n+k}
\]
with a quasi-homogeneous action given by a positive weight vector $\xi = (w_1,\ldots, w_{n+k})$ such that
\[
f_i(t\cdot x) = f_i(t^{w_1}x_1,\ldots,t^{w_{n+k}}x_{n+k}) = t^{d_i} f_i(x).
\]
We assume each $d_i > 1$ and $d_1\leq d_2 \leq \cdots \leq d_k$.
This action generates the action of a complex torus $T^C$ on $\bC^{n+k}$ that leaves $V_0$ invariant.
Let $T \subset T^C$ be the maximal compact torus.
We notate $F : \bC^{n+k} \to \bC^k$ to be $F(x) = (f_1(x),\ldots,f_k(x))$ and label fibers $V_t = F^{-1}(t)$:
\[
V_t = \{f_1(x)-t_1 = \cdots = f_k(x)-t_k = 0 \}\subset \bC^{n+k}.
\]
There is also a natural action of $\edit{\bC^*}$ on the codomain of $\bC^k$, which we will also notate using $t \cdot s$ for $s = (s_1,\ldots,s_k)$ by $\lambda \cdot s = (\lambda^{d_1}s_1,\ldots,\lambda^{d_k}s_k)$.
If $d_1 = d_k$, then the result follows similarly from \cite{GaborPaper}, so throughout we will assume that $d_1 \neq d_k$. 
Since $V_0$ is assumed to be a Calabi-Yau cone, there is a nowhere-vanishing holomorphic $n$-form $\Omega_V$ on $V_0 \setminus \{0\}$ given as
\[
\Omega_{V} = \frac{dx_1\wedge \cdots \wedge dx_{n+k}}{df_1\wedge \cdots \wedge df_k},
\]
which is a notation for the global expression defined locally as 
\[
\Omega_{V} = \frac{dx_{k+1}\wedge \cdots \wedge dx_{n+k}}{\pa_{x_1}f_1\cdots \partial_{x_k}f_k}
\]
where the denominator does not vanish, and similar expressions for other regions where $\pa_{x_{i_1}}f_{1}\cdots \pa_{x_{i_k}}f_{k}$ is non-zero with $i_k$ distinct.
$V_0 \setminus \{0\}$ has a symplectic form $\omega_{V_0}$, which is a Ricci-flat K\"ahler cone metric such that 
\begin{equation*}\label{eqn:CYmetricToOmega}
\omega_{V_0}^n = (\sqrt{-1})^{n^2}\Omega_{V} \wedge \ol{\Omega}_{V}.
\end{equation*}
In order for this to be true, we must have a topological criterion stating that $\Omega$ has degree $n$ under the action, or equivalently 
\begin{equation}\label{eqn:WeightDegreeRestraint}
\sum w_j - \sum d_i = n. 
\end{equation}
We further assume $V_0$ is non-degenerate and is not contained in a codimension-1 variety.
\begin{lemma}
If $V_0$ does not lie in a \edit{hyperplane} properly contained in $\bC^{n+k}$, then $d_1 > 2$.
\end{lemma}
\begin{proof}
Let $w_{\mr{min}}$ be the smallest weight. 
Since we assume that $V_0$ admits a Calabi-Yau metric, by the Lichnerowicz obstruction of Gauntlett-Martelli-Sparks-Yau \cite{Gauntlett_2007}, $w_{\mr{min}} > 1$. 
If $w_{\mr{min}}$ is equal to 1, then $V_0$ is contained in a smaller dimensional hypersurface. 
Each $d_i$ must be at least $2w_{\mr{min}}$, showing that each is greater than 2.
\end{proof}

By Sard's theorem, we can choose some point $p = (p_1,\ldots,p_k)$ normalized to $|p| = 1$ using the scaling action such that $V_p$ is smooth.
\begin{definition}\label{def:ell&p'}
We define $\ell$ to be the maximal index such that $d_1 = d_2 = \cdots = d_\ell < d_{\ell + 1} \leq \cdots \leq d_k$ and define the point $p' = (p_1,\ldots,p_{\ell},0\ldots,0)$.
We say that $\cX$ is smooth at infinity if the fiber $V_{p'}$ is smooth.
\end{definition}
In many cases, there exists a smooth fiber of this form, or a small perturbation of the polynomials exists that preserves $K$-stability, (which is equivalent to having a Calabi-Yau metric as per Collins-\Sz \cite{CollinsTristan2019SmaK}), and has a smooth fiber of this form. 
However, not all complete intersections admit such a smooth fiber, or even perturbations that admit a smooth fiber of this form.
For example, for $f_1 = z_1^2 + z_2z_3 + z_4^3, f_2 = z_1z_4^2 + z_2^2 + z_3^2z_4 + z_5^3$ with weights $\xi = (27,63/2, 45/2,18,21)$ and of degrees $d_1 = 54$ and $d_2 = 63$, no fiber $V_{(t_1,0)} = \{f_1 - t_1 = f_2 = 0\}\subset \bC^5$ is smooth and no perturbation of the polynomials fixes this singularity either by smoothing it, or perturbing it off the vanishing locus.
It would be interesting to know if the smooth at infinity assumption could be dropped.

We define the total space $\cX$ as swept out by the set of fibers $V_{zp}$,
\[
\cX =\{f_1(x) - zp_1 = \cdots = f_k(x) - zp_k = 0\}\subset\bC^{n+k+1}.
\]
$\cX$ is a complete intersection of dimension $n+1$ in the ambient space $\bC^{n + k+1}$ with variables $(x_1,\ldots,x_{n+k},z)$.
We require at least one fiber to be smooth, which we label $V_1$.
The singular fibers cut out an analytic subvariety of $\bC$, and therefore, there are either finitely many or every fiber is singular.
The assumption that $V_1$ is smooth ensures that outside a sufficiently large compact set, all fibers are smooth.

On $\cX$, we will construct a Calabi-Yau metric based on gluing a Calabi-Yau metric on each smooth fiber $V_t$ with the product Calabi-Yau metric of $V_0 \times \bC$.
Let $X_t = V_t \times \bC$ with $X_0$ and $X_{p'}$ being of primary interest.
There is a nowhere-vanishing holomorphic $(n+1)$-form $\Omega$ on $\cX$ and each $X_t$ is defined using notation from above as
\[
\Omega = \frac{dz \wedge dx_1 \wedge \cdots \wedge dx_{n+k}}{df_1\wedge \cdots \wedge df_{k}},
\]
which can be explicitly computed as $\Omega = dz \wedge \pi_x^* \Omega_V$ with $\pi_x : \bC^{n+k+1}\to \bC^{n+k}$ the projection onto the $x$-coordinates.

Let $\omega_{V_0} = \iddbar r^2$ for the cone radius function $r$ on $V_0$.
We define a scaling action $G_t:\bC^{n+k+1} \to \bC^{n+k+1}$ given by
\begin{equation}\label{eqn:GtMaps}
G_t(z,x_1,\ldots,x_{n+k}) = (tz,t^{w_1}x_1,\ldots,t^{w_{n+k}}x_{n+k}),
\end{equation}
which acts by the weight vector $\xi$ on $\bC^{n+k}$ and linearly on $z$.
The maps $G_t$ are the homothetic transformations on $V_0$, the transformations generated by its Reeb field.
The fibers $V_{(t_1,\ldots,t_k)}$ sweep out $\bC^{n+k}$ as $t = (t_1,\ldots,t_k)$ vary through all possible values.
Note that $\cX$ is realized as the linear slice where $t = zp$ for $z \in \bC$ to form $\cX$.
When $d_1 \neq d_k$, $G_t(\cX) \not\subset \cX$; this scaling action does not preserve the fibration structure of $\cX$.
It is for this reason that we need to construct a global function that restricts to a Calabi-Yau metric on each fiber, a construction not needed when the fibration is the trivial one over $\bC^*$, as in \cite{GaborPaper}.

We define a radius function $R$ on the ambient space $\bC^{n+k+1}$ which restricted to $V_0$ is uniformly equivalent to the radius function $r$.
We can define $R$ to be 1 on the unit sphere $S^{2(n+k)+1} \subset\bC^{n+k+1}$ and defining it elsewhere by enforcing that it has degree 1 under the map $G_t$.
The function $r$ can be extended arbitrarily to $\bC^{n+k+1}$ by first extending it smoothly on the sphere $R = 1$ and then to other points by similarly enforcing that it is homogeneous with degree 1 under $G_t$.
The extended function will still be notated as $r$.
\begin{proposition}\label{prop:CYFiberMetric}
Let $V_t$ be a smooth fiber.
The form $\omega = \iddbar r^2$ on $V_t$ is a well-defined metric outside a compact set, for $V_t \cap \{R > A\}$ for some $A \gg 0$, and its Ricci potential
\[
h = \log \dfrac{(\iddbar r^2)^{n}}{(\sqrt{-1})^{n^2} \Omega_{V} \wedge \ol{\Omega}_{V}}
\]
satisfies the decay estimate $|\nabla^i h|_{\omega} = O(R^{-d_1-i})$ as $R \to \infty$ with this norm measured using $\omega$.
\end{proposition}
\begin{proof}
Let $K \gg 0$ be a large constant and examine the annular region $\frac{K}{2} < R < 2K$.
We will rescale the metric by $K^{-1}$ and label the rescaled coordinates as 
\[
\tilde{x} = K^{-1}\cdot x,\quad \tilde{r} = K^{-1}r,\quad \tilde{R} = K^{-1}R,
\]
so the metric in these new coordinates can be expressed as $\tilde{\omega} = K^{-2}\omega = \iddbar \tilde{r}^2$.

In these rescaled coordinates, the algebraic expression for $V_t$ is given by $F(\tilde{x}) = K^{-1} \cdot t$.
Expanded, this is expressed as 
\[
\{f_1(\tilde{x}) - K^{-d_1}t_1 = \cdots= f_1(\tilde{x}) - K^{-d_k}t_k = 0\},
\]
or $V_{K^{-1}\cdot t}$ in the concise notation.
Specifically, on the annulus defined by $\{\frac{1}{2} < \tilde{R} < 2\}$, the spaces $V_{K^{-1}\cdot t}$ converge to $V_0$ in the $C^\infty$ topology uniformly on compact sets since each $d_i > 1$. 
Furthermore, $\tilde{r}^2$ is a function of only $\tilde{x}$, independent of $K$ since $r$ is quasi-homogeneous with respect to the scaling action.

The implicit function theorem gives a cover of the annular region $\{\frac{1}{2} < \tilde{R} < 2\}$ in $V_{K^{-1}\cdot t}$ by finitely many coordinate patches $B_i$ such that $V_{K^{-1}\cdot t}$ is locally expressible as a hyperplane in each chart.
The implicit function theorem provides holomorphic functions $g_i : B_i \to \bC^{n+k}$ such that $g_i(V_{K^{-1}\cdot t})$ is contained in $V_0$, and $g_i(y) = y + O(K^{-d_1})$.
The decay term is limited by the smallest degree $d_1$, as will be the case in all the following arguments.
We use these local expressions to compare the rescaled potential to its pullback, and see that it has decay given by $K^{-d_1}$.
We can observe that $\tilde{r}^2 - g_i^*(\tilde{r}^2)$ must have growth decay $O(K^{-d_1})$, so once $K$ is large enough, the form $\tilde{\omega}$ is indeed positive definite, proving the first claim.
It is further uniformly equivalent to $\iddbar \tilde{R}^2$, which is a cone metric on $\bC^{n+k+1}$ from He-Sun \cite[Lemma 2.2]{HeWeiyong2016FcaS}.

For the second part of the proposition, we want to examine the Ricci potential and compare $\Omega$ to its pullback along each implicit function $g_i$.
Since by assumption $\omega_{V_0}^n = \sqrt{-1}^{n^2}\Omega \wedge \overline{\Omega}$, the difference has bounded growth given by $\Omega_V - g_i^*(\Omega_V) = O(K^{-d_1})$.
The Ricci potential is invariant under scaling from equation \ref{eqn:WeightDegreeRestraint}.
This gives the decay 
\[
|\nabla^i h|_{K^{-2}\iddbar R^2} = O(K^{-d_1})
\]
finishing the result since for $K \gg 0$ sufficiently large, on $V_0$ the two forms $\iddbar R^2$ and $\iddbar r^2$ give rise to uniformly equivalent metrics.
When we instead measure using the unscaled metric $\iddbar r^2$, there is an added factor of $K^{-1}$ for each derivative, so $|\nabla^i h|_{\iddbar r^2} = O(K^{-d_1-i})$.
\end{proof}
The form $\omega = \iddbar r^2$ is only a positive definite metric for $R \gg 0$, so we glue in the form $C(1 + |x|^2)^\alpha$ on the bounded region where $\omega$ is not positive definite and piece the two together to make a globally well-defined metric.
For $C$ sufficiently large and $\alpha > 0$, this is greater than ${r}^2$ on a large ball around the origin of radius $A$ from Proposition \ref{acybound}.
The form $\iddbar r^2$ becomes a well-defined metric outside this region.
We take a regularized maximum of the two forms above to have a global metric on $V_t$, a construction found in Demailly \cite[\textsection 5.E]{Demailly}.
The results of Conlon-Hein \cite{Conlon_2013} say that this new metric, which is asymptotic to a Calabi-Yau metric with sufficient growth decay, can be perturbed to a Calabi-Yau metric defined by $\eta_t = \iddbar \phi_t$.
Furthermore, this metric is unique by \cite[Theorem 2.1]{Conlon_2013}.
In particular, $\phi_t$ is invariant under the action of the compact torus $T$ generated by the $\bC^*$-action.
Notably, the value of $\phi_t(z^{-1/d_1}\cdot x)$ is well-defined regardless of the choice of branch cut used to define the fractional exponent, which will be used to define the asymptotically Calabi-Yau metric on all of $\cX$.
Theorem 2.1 of \cite{Conlon_2013} gives some $c > 0$ (if $d_1> 3$, then even $c>1$ can be chosen), constants $C_i$ such that for $x$ in the annular region $\frac{1}{2} < R < 2$, the following estimate holds for $\lambda > 1$:
\begin{equation}\label{CHbound}
\left| \nabla^i\left(r^2 - \lambda^{-2}\phi_t(\lambda\cdot x)\right)\right|_{\iddbar R^2} < C_i\lambda^{-2-c}. 
\end{equation}
As constructed above, the Calabi-Yau potentials $\phi_t$ on $V_t$ arise as perturbations of the same global potential $r^2$ defined on $\cX$ as restrictions to each fiber $V_t$.
By the uniqueness results of Conlon-Hein \cite{Conlon_2013}, the functions $\phi_t$ assemble together to form a smooth function $\Phi$.
Let $\cX_{\mr{sm}}$ be the union of all fibers $V_{t}$ that are smooth.
In $\bC^{n+k+1}$, $\Phi$ is defined over the union of all smooth fibers $V_t$, notably over all fibers over a neighborhood of $p'$ and $\cX_{\mr{sm}}$.
The above is summarized in the following corollary. 
\begin{corollary}\label{cor:CYMetricOnFiber}
There exists a smooth function $\Phi$ defined on $\cX_{\mr{sm}}$ whose restriction to each $V_t$ satisfies 
\[
\Phi\vert_{V_t} = \phi_t, \quad (\iddbar \phi_t)^n = (\sqrt{-1})^{n^2}\Omega_V \wedge \ol{\Omega}_V
\]
which is the Calabi-Yau metric asymptotic to $r^2$ in that it satisfies the decay conditions of equation \ref{CHbound}.
\end{corollary}
We note that we will only be using $\Phi$ in a neighborhood around $V_{p'}$ corresponding to fibers $V_{tp}$ for $|t| \gg 1$ sufficiently large where all the fibers are smooth by assumption, so the form $\iddbar \Phi$ will be well-defined.
\section{The Approximate Solution}
\label{2ACYmetric}
In this section, we produce a metric on $\cX$ which is asymptotically Calabi-Yau with sufficient decay of the Ricci potential such that we will be able to later perturb it to be Ricci-flat outside a compact set.
The metric will be constructed by gluing model Calabi-Yau metrics on different regions near infinity.
The model spaces will be products $X_0 = \bC \times V_0$ and $X_{p'} = \bC \times V_{p'}$ with metrics $\iddbar(|z|^2 + r^2)$  and $\iddbar(|z|^2 + \phi_{p'})$, respectively.
Since each fiber itself is asymptotic to $V_0$, as we approach infinity mainly in the fiber direction $R$, the model is $X_0$.
When we approach infinity mainly in the $z$-direction, the model is $X_{p'}$.
These model spaces are realized as cones in $\bC^{n+k+1}$ with homothetic scalings given by the maps $G_t$ in equation \ref{eqn:GtMaps} using the action with weights of $1$ on the $z$ variable and $\xi$ on the $x$ variables.

Let $\gamma_1$ be a monotonic cutoff function such that $\gamma_1(x) = 0$ for $x < 1$ and $\gamma_1(x) = 1$ for $x > 2$. 
Let $\gamma_2 = 1- \gamma_1$ be its complement.
We define a new radius function on the total space $\cX$ as $\rho^2 = R^2 + |z|^2$.
We define a metric
\begin{equation}\label{eqn:omegaDefn}
\omega = \iddbar(|z|^2 + \gamma_1(R/\rho^\alpha)r^2 + \gamma_2(R/\rho^\alpha)|z|^{\frac{2}{d_1}}\Phi(z^{-\frac{1}{d_1}}\cdot x))
\end{equation}
for a small constant $\alpha$ to be chosen. 
Notating $\Psi = \Phi - r^2$ on $\cX$ where $\rho \gg 0$ is sufficiently large, the metric can be expressed succinctly as 
\[
\omega = \iddbar(|z|^2 + r^2 + \gamma_2(R/\rho^\alpha)|z|^\frac{2}{d_1}\Psi(z^{-\frac{1}{d_1}}\cdot x)),
\]
emphasizing that this metric is asymptotic to $\iddbar(|z|^2 + r^2)$ since $\Psi$ decays rapidly per equation \ref{CHbound}.
The main result of this section is to show sufficient decay of the Ricci potential $h$.
\begin{proposition}
\label{acybound}
For some $\alpha \in (\frac{1}{d_1},1)$, the form $\omega$ from equation \ref{eqn:omegaDefn} defines a metric on $\cX$ outside a compact set, and for suitable constants $\kappa, C_i > 0$, there exists $\delta < \frac{2}{d_1}$ such that the Ricci potential $h$ has decay
\[
|\nabla^i h| = \begin{cases}
C_i \rho^{\delta-2-i},& R > \kappa \rho\\
C_i \rho^{\delta}R^{-2-i},& R \in(\kappa^{-1}\rho^{\frac{1}{d_1}},\kappa \rho )\\
C_i \rho^{\delta - \frac{2}{d_1} - \frac{i}{d_1}},& R < \kappa^{-1}\rho^{\frac{1}{d_1}}.
\end{cases}
\]
\end{proposition}
\begin{proof}
We split the ambient space into five regions to perform the analysis depending on from which direction we approach infinity, (i.e., the asymptotic relationship between $R$ and $z$).
The approach follows similarly to \Sz \cite{GaborPaper} in the first two regions.
The latter regions differ because of the non-trivial fibration structure.
\smallskip
\noindent\textbf{Region \RNum{1}:} In this region, let $R > \kappa \rho$ for some $0 < \kappa < 1$ where we are far away from the \textit{cone point}.
Here, each $V_t$ is approximated well by $V_0$, so the appropriate model space is $X_0$.
We define rescaled coordinates 
\[
\tilde{z} = D^{-1}z,\quad \tilde{x} = D^{-1}\cdot x,\quad \tilde{r} = D^{-1}r
\]
and examine $D^{-2}\omega$, the properly rescaled metric in these coordinates.
In this region, $\gamma_1 = 1$ and for $D \gg 0$ sufficiently large, the metric is simply $\iddbar(|z|^2 + r^2)$.
We can re-express $\cX$ in these rescaled coordinates as 
\[
\cX = \{f_1(\tilde{x}) - D^{1-d_1}\tilde{z}p_1 = \cdots = f_k(\tilde{x}) - D^{1-d_k}\tilde{z}p_k=0\},
\]
and we compare to $X_0 = V_0 \times \bC$ with metric $\iddbar(|z|^2 + r^2)$.
The error induced by the approximation by this variety is $D^{1-d_1}$, as this is the limiting term since it is the greatest exponent, similar to in Proposition \ref{prop:CYFiberMetric} by using the implicit function theorem.

This gives the decay bound $|\nabla^i h|_{D^{-2}\omega} \leq C_i D^{1-d_1}$, and by using the unscaled metric $\omega$, there is a term of $D^{-1}$ for each derivative giving $|\nabla^i h|_{\omega} \leq C_i D^{1-d_1-i}$.
We conclude that in this region, $\delta$ can be chosen such that $\delta > 3-d_1$.
Since $d_1 \geq 2$, $\delta$ can be chosen to satisfy $\delta < \frac{2}{d_1}$, and if $d_1 \geq 3$, $\delta$ can be made negative.

\smallskip
\noindent\textbf{Region \RNum{2}:} We let $R \in (K/2,2K)$ for some $K < \kappa \rho$ and $K/2 > 2\rho^\alpha$.
This region is similar to Region \RNum{1} in that the appropriate model space is still $X_0$ as $R$ is sufficiently greater than $|z|$.
This region still has $\gamma_1 = 1$ and therefore the same metric $\omega = \iddbar(|z|^2 + r^2)$.
We focus the analysis around some basepoint $z_0$ such that $|z-z_0| < K$.
We rescale the coordinates in a similar manner, but centered at $z_0$:
\[
\tilde{x} = K^{-1}\cdot x,\quad \tilde{z} = K^{-1}(z-z_0),\quad \tilde{r} = K^{-1}r.
\]
In these new coordinates, the rescaled metric is expressed as $K^{-2}\omega = \iddbar(|\tilde{z}|^2 + \tilde{r}^2)$, and the equations defining $\cX$ in these coordinates are 
\[
\{f_1(\tilde{x}) -K^{-d_1}(K\tilde{z} + z_0)p_1 = \cdots = f_k(\tilde{x}) -K^{-d_k}(K\tilde{z} + z_0)p_k= 0\}.
\]
Since $|\tilde{z}|,\ p_i < 1$, the procedure is as above, but with error term $K^{-d_1}D$, giving the decay bound 
\[
|\nabla^i h|_\omega \leq C_i DK^{-d_1-i}.
\]
Because $d_1 > 2$ and $K > 4\rho^\alpha$, we compute 
\begin{equation}\label{eqn:K-dDbound}
DK^{2-d_1}K^{-2-i} < CD^{1 + \alpha(2-d_1)}K^{-2-i} < C\rho^{1 + \alpha(2-d_1)}R^{2-i},
\end{equation}
so in order to have $\delta > 1 + \alpha(2-d_1)$, we let $\alpha$ be sufficiently close to 1, so that $\delta > 3-d_1$.
This gives the same constraint as in the prior region.
We must have $\alpha > \frac{1}{d_1}$, so this is satisfied.

\smallskip
\noindent\textbf{Region \RNum{3}:} This is the gluing region where $R \in (K/2,2K)$ and $K \in (\rho^\alpha, 2\rho^\alpha)$ and $\rho \in (D/2,2D)$.
In this region, $|z| \sim D$ and both $\gamma_1$ and $\gamma_2$ are non-zero, so the metric will be the most complicated, having both terms.
The model space will still be $X_0$ as the Conlon-Hein bound from equation \ref{CHbound} will allow us to model the region with sufficient decay as the terms coming from $\phi_t$ approximate the geometry of the cone $X_0$.
We use the same rescaling as in the previous region
\[
\tilde{x} = K^{-1}\cdot x,\quad \tilde{z} = K^{-1}(z-z_0),\quad \tilde{r} = K^{-1}r.
\]
Under this change of coordinates, the rescaled metric is expressed as
\[
K^{-2}\omega = \iddbar(|\tilde{z}|^2 + \gamma_1\tilde{r}^2 + \gamma_2K^{-2}|K\tilde{z} + z_0|^{-2}\Phi((K\tilde{z}+z_0)^{-\frac{1}{d_1}}K\cdot \tilde{x})).
\]
In these rescaled coordinates, the derivatives of $\gamma_1$ and $\gamma_2$ are bounded and by applying the Conlon-Hein bound from equation \ref{CHbound}, we will demonstrate the desired decay rate.
In these coordinates, the manifold is expressed as
\[
\cX = \{f_1(\tilde{x}) - K^{-d_1}(K\tilde{z} + z_0)p_1 = \cdots = f_k(\tilde{x}) - K^{-d_k}(K\tilde{z} + z_0)p_k=0\}
\]
as in the prior region.
We apply the Conlon-Hein estimate \ref{CHbound} to compute the decay as 
\[
\nabla^i\left[ K^{-2}|K\tilde{z} + z_0|^{\frac{2}{d_1}}\Phi((K\tilde{z} + z_0)K\cdot \tilde{x})-\tilde{r}^2\right] = O((K^{-1}D^{\frac{1}{d_1}})^{2+c})
\]
since $K\tilde{z} + z_0 = z$ is of order $D$ as specified in the region.
There are two error terms of $K^{-d_1}D$ as in Region \RNum{2}, and the new error term $(K^{-1}D^{\frac{1}{d_1}})^{2+c}$.
The first term was already shown to satisfy the desired decay, so we must show it for the new error term given by the Conlon-Hein estimate \ref{CHbound}.
In this region, $K \sim D^\alpha$, so this term can be bounded
\[
(K^{-1}D^{\frac{1}{d_1}})^{2+c} < CD^{\frac{2+c}{d_1}-c\alpha}K^{-2}
\]
by replacing $K^{-c}$ with $D^\alpha$.
This expresses the new error term in the desired form, and we must demonstrate that $\delta$ can be chosen such that $\delta  > \frac{2+c}{d_1}-c\alpha$.
Let $\alpha$ be $1 - \varepsilon$, and then choose $\delta$ such that 
\[
\delta > \frac{2+c}{d_1} - c\alpha = \frac{2}{d_1} + \frac{c\alpha - d_1c}{d_1}.
\]
Since $d_1 > 2$, this is less than $\frac{2}{d_1}$.
Therefore, $\delta$ can be chosen to be less than $\frac{2}{d_1}$.
Furthermore, if $d_1 > 3$, $c$ can be given as $c > 1$, allowing a choice of $\delta$ to be negative.

\smallskip
\noindent\textbf{Region \RNum{4}:} In this region, the model space is $X_{p'} = \bC \times V_{p'}$ with $p' = (p_1,\ldots,p_\ell,0,\ldots,0)$.
The natural rescaling will produce a model space of $X_{t}$ with $t\to p'$.
Sufficiently far away, the model spaces will be smooth and of the form $V_{p' + \ve}$ and can therefore be approximated by $V_{p'}$ with sufficiently small error.

We use the range $R \in (K/2,2K)$, for $K \in (\kappa^{-1}\rho^\frac{1}{d_1},\rho^\alpha/2)$ and $\rho \in (D/2,2D)$.
In this region, $|z| \sim D$ as before.
However, $|z|$ is growing sufficiently faster than $r$, so the model space can no longer be $X_0$.
We use the same change of coordinates as above with a similarly defined $z_0$:
\[
\tilde{x} = K^{-1}\cdot x,\quad \tilde{z} = K^{-1}(z-z_0),\quad \tilde{r} = K^{-1}r.
\]
Since $\gamma_2 = 1$, the metric looks like a product of the flat metric on $\bC$ and $\iddbar\phi_t$.
The equations for $\cX$ are
\[
\cX = \{f_1(\tilde{x}) - K^{-d_1}(K\tilde{z} + z_0)p_1 = \cdots =f_k(\tilde{x}) - K^{-d_k}(K\tilde{z} + z_0)p_k =0\}.
\]
We compare this to the model space of $\bC \times V_{\hat{p}}$, for $\hat{p} = K^{-1}\cdot z_0p = (K^{-d_1}z_0p_1,\ldots,K^{-d_k}z_0p_k)$ which is sufficiently close to $p'$ so that $V_{\hat{p}}$ is smooth by the assumption that $V_{p'}$ is smooth.
The error term induced by this comparison is $K^{1-d_i}\tilde{z}$ in each component, with the maximal error where $i = 1$ of $K^{1-d_1}$.
The model metric is
\[
\omega_{X_{\hat{p}}} = \iddbar(|\tilde{z}|^2 + K^{-2}|z_0|^{\frac{2}{d_1}}\Phi(K|z_0|^{-\frac{1}{d_1}}\cdot \tilde{x})),
\]
and the metric on $\cX$ on the region is
\[
\omega = \iddbar(|\tilde{z}|^2 + K^{-2}|K\tilde{z} + z_0|^{\frac{2}{d_1}}\Phi((K\tilde{z} + z_0)^{-\frac{1}{d_1}}K\cdot\tilde{x})).
\]
We can use the homogeneity of $r$ to examine the difference
\[
E = \iddbar(K^{-2}|K\tilde{z} + z_0|^{\frac{2}{d_1}}\Phi((K\tilde{z} + z_0)^{-\frac{1}{d_1}}K\cdot\tilde{x}) - K^{-2}|z_0|^{\frac{2}{d_1}}\Phi(K|z_0|^{-\frac{1}{d_1}}\cdot \tilde{x})),
\]
which is the same as replacing $\Phi$ with $\Psi = \Phi - r^2$, which satisfies the decay criteria of being in the weighted space $C^\infty_{-c}(V_t)$, which is the existence of constants given by equation \ref{CHbound}.
The homogeneity of $r$ means all the $r^2$ terms cancel.

To finish, we need to bound this difference of potentials.
The function $\Psi$ is a function of both $z$ and $x = (x_1,\ldots, x_{n+k})$ and we denote $\Psi(z,x) = \Psi_{z}(x)$.
Since the points are very close, an estimate given by the gradient is
\[
\Psi_{q_1}(x_1) - \Psi_{q_2}(x_2) \leq \bigg\vert\frac{\p \Psi}{\p q}\bigg\vert|q_1-q_2| + \bigg\vert\frac{\p \Psi}{\p x}\bigg\vert|x_1-x_2| .
\]
Therefore, $\big\vert\frac{\p \Psi}{\p q}\big\vert < C$ since $\Phi(q,x)$ is smooth in both variables on a small disk around $p'$ with bounded gradient.

Since $q_1-q_2 = z^{-\frac{1}{d_1}}\cdot p - K^{-1}\cdot z_0p$, this is of order $\max(D^{-1/d_1}, K^{-d_1}D)$, and the first term fits the decay criterion.
The term $K^{-d_1}D$ can be bounded as
\begin{equation}\label{eqn:K-dDBound2}
   K^{-d_1}D = K^{-2}K^{2-d_1}D \leq CK^{-2}D^{\alpha(2-d_1) + 1}, 
\end{equation}
so we need $\alpha(2-d_1) + 1 < \delta$.
We consider $\alpha(2-d_1) + 1 < \frac{2}{d_1}$ and choose $\alpha$ sufficiently close to $1$, $\alpha > \frac{1}{d_1}$ to get the bound.
Therefore, $\delta$ can be chosen such that $\alpha(2-d_1) + 1 < \delta < \frac{2}{d_1}$ as necessary, as shown above in equation \ref{eqn:K-dDbound}.
In this region, $K < CD^{\alpha}$ so the first term is bounded by $D^{-1/d_1} < CK^{-2}K^2 D^{-1/d_1} < CK^{-2}D^{2\alpha - 1/d_1}$.
Since $\alpha > 1/d_1$, $\delta$ can be chosen to be less than $2/d_1$ as desired.
Therefore, we only need to examine the term $E$ as above.

We can Taylor expand these terms, and given that $|\tilde{z}| < 1$, $|z_0| \sim D$ and $K \ll D$, the growth is given by
\[
K(K\tilde{z} + z_0)^{-\frac{1}{d_1}} = z_0^{-\frac{1}{d_1}}K(1 + O(KD^{-1}))
\]
and we apply this with the estimate from equation \ref{CHbound} to get
\[
|\nabla^i E| < C_i (|z_0|^{-\frac{1}{d_1}}K)^{-2-c}KD^{-1} = O(K^{-1-c}D^{\frac{2+c}{d_1}-1}),
\]
and we combine this with the first error term of $K^{1-d_1}$, which will force the choice of $\delta$ to satisfy
\[
K^{1-d_1} + K^{-1-c}D^{\frac{2+c}{d_1}-1} < CD^\delta K^{-2}.
\]
Suppose that $d_1 > 3$ and thus $c > 1$, so this can be bounded as
\[
K^{1-d_1} = K^{3-d_1}K^{-2} < C D^{\frac{3}{d_1}-1}K^{-2}
\]
for the first term, and likewise
\[
K^{-1-c}D^{\frac{2+c}{d_1}-1} = (KD^{-\frac{1}{d_1}})^{1-c}D^{\frac{3}{d_1}-1}K^{-2}
\]
and since $\frac{3}{d_1} - 1 < 0$, $\delta$ can be chosen to be negative for these assumptions.

Now we suppose that $d_1 > 2$ only and thus we only know that $c > 0$, giving the estimate
\[
K^{-1-c}D^{\frac{2+d_1}{d_1}-1} = \underbrace{(KD^{-\frac{1}{d_1}})^{-c}}_{< C}\underbrace{(KD^{-1})}_{<D^{\varepsilon}}D^{\frac{2}{d_1}}K^{-2},
\]
which shows that $\delta$ can be chosen to satisfy $\delta < \frac{2}{d_1}$ as desired.

We want our model space to be consistently $X_{p'}$, not  $X_{\hat{p}}$, which varies in $z_0$.
We can compare these two spaces and see that the error induced will be sufficiently small.
Their equations are expressed as 
\[
X_{\hat{p}} = \{f_i(x) = K^{-d_i}z_0p_i\}\text{ compared to } X_{p'} = \{f_i(x) = p_i , f_j(x) = 0: i \leq \ell, j > \ell\}.
\]
The error term is computed by comparing $K^{-d_i}z_0p_i$ to $p_i$ for $i \leq \ell$ or $K^{-d_j}z_0$ for $j > \ell$.
Since there are bounds $C' D^{\frac{1}{d_1}}< K < CD^\alpha$ for $\alpha \in (\frac{1}{d_1},1)$, $K^{-d_j}D$ has sufficient decay for the region approximating this going to 0, since $d_j > d_1$ and $K\sim D^{\alpha}$.
The maximal error term of $K^{-d_1}D$ fits the decay as shown above in equation \ref{eqn:K-dDBound2}.
Furthermore, as above, if $d_1 > 3$, then $\delta$ can be chosen to be negative.

\smallskip
\noindent\textbf{Region \RNum{5}:} In this final region, $R < 2\kappa^{-1}\rho^{\frac{1}{d_1}}$ and $\rho \in (D/2,2D)$.
We perform a similar rescaling, choosing $z_0$ close to $z$ and we rescale by $|z_0|^{-\frac{1}{d_1}}$:
\[
\tilde{z} = z_0^{-\frac{1}{d_1}}(z-z_0),\quad \tilde{x} = z_0^{-\frac{1}{d_1}}\cdot x,\quad \tilde{r} = |z_0|^{-\frac{1}{d_1}}r.
\]
Under this rescaling, both $|\tilde{z}|,\tilde{r} < C$ are bounded.
The metric in these coordinates is expressed as
\[
|z_0|^{-\frac{2}{d_1}}\omega = \iddbar(|\tilde{z}|^2 + |z_0|^{-\frac{2}{d_1}}|z_0^{\frac{1}{d_1}}\tilde{z} + z_0|^{\frac{2}{d_1}}\Phi(z_0^{\frac{1}{d_1}}(z_0^{\frac{1}{d_1}}\tilde{z} + z_0)^{-\frac{1}{d_1}}\cdot\tilde{x})),
\]
and the equations expressing $\cX$ are 
\[
\cX = \{f_1(\tilde{x}) = z_0^{-\frac{d_1}{d_1}}(z_0^{\frac{1}{d_1}}\tilde{z} + z_0)p_1 = \cdots = f_k(\tilde{x}) = z_0^{-\frac{d_k}{d_1}}(z_0^{\frac{1}{d_1}}\tilde{z} + z_0)p_k= 0\}.
\]
Separating out the terms with variable $\tilde{z}$, we compare to the \edit{model variety $X_{\hat{p}}$ holding $z_0$ constant} cut out by equations
\[
X_{\hat{p}} = \{f_1(\tilde{x}) = z_0^{1-\frac{d_1}{d_1}}p_1 = \cdots=f_k(\tilde{x}) = z_0^{1-\frac{d_k}{d_1}}p_k  = 0\},
\]
which is smooth because $\hat{p} = (p_1,\ldots,p_\ell, z_0^{1-d_{\ell+1}/d_1}p_{\ell +1},\ldots,z_0^{1-d_k/d_1}p_{k})$ is sufficiently close to $p'$ whose fiber $V_{p'}$ is smooth by assumption, similarly to above.
We endow $\bC \times V_{\hat{p}}$ with the metric $\omega = \iddbar(|\tilde{z}|^2 + \phi_{\hat{p}}(\tilde{x}))$, noting that for $D \gg 0$ sufficiently large, $|t| < 1$ since $|p| = 1$.
To compare these two metrics, we estimate
\[
E = |z_0|^{-\frac{2}{d_1}}|z_0^{\frac{1}{d_1}}\tilde{z} + z_0|^{\frac{2}{d_1}}\Phi(z_0^{\frac{1}{d_1}}(z_0^{\frac{1}{d_1}}\tilde{z} + z_0)^{-\frac{1}{d_1}}\cdot\tilde{x}) - \Phi(\tilde{x})
\]
and we Taylor expand $z_0^{\frac{1}{d_1}}(z_0^{\frac{1}{d_1}}\tilde{z} + z_0)^{-\frac{1}{d_1}} = 1 + O(D^{\frac{1}{d_1}-1})$.
Because $\tilde{x}$ and $z_0^{\frac{1}{d_1}}(z_0^{\frac{1}{d_1}}\tilde{z} + z_0)^{-\frac{1}{d_1}}\cdot x)$ lie on different fibers, namely $p'$ and $\hat{p}$, the error term induced by the difference of their rescaled potentials from $\Phi$ can be compared as $\phi_{p'} - \phi_{\hat{p}}$, which is sufficiently small to fit the desired decay because $|p' - \hat{p}|$ can be made arbitrarily small.
As above, the derivative is bounded, so we need for $z_0^{1-\frac{d_i}{d_1}} - z^{1-\frac{d_i}{d_1}}$ to decay as $D^{-\varepsilon}$ for some $\varepsilon > 0$.
Indeed, this decays like $D^{-1-\varepsilon'}$ for $\varepsilon = 1 - \frac{d_{\ell+1}}{d_1}$ where $d_{\ell+1}$ is minimal with $d_1 = d_\ell < d_{\ell +1}$.
Here, $\delta$ can even be negative and this will not be the limiting term in the estimates.

This leaves only the error term of $E$ as before.
$E$ can be bounded by the same techniques as before, giving $|\nabla^i E| < CD^{\frac{1}{d_1}-1}$, so $\delta$ must be chosen to satisfy $D^{\frac{1}{d_1} - 1} < CD^{\delta - \frac{2}{d_1}}$.
If $d_1 \geq 2$, then $\delta$ can be chosen such that $\delta < \frac{2}{d_1}$, and if $d_1 > 3$, then $\delta$ can further be negative, concluding the proof.

Similarly to the previous region, instead of working on the varying space of $X_{\hat{p}}$ which depends on $z_0$, it will help to work on the fixed space $X_{p'}$. 
As $|z_0| \to \infty$, $\hat{p} = z_0^{-1/{d_1}} \cdot z_0p \to p'$, demonstrating this as the limiting fiber.
The error term in moving from $V_{\hat{p}}$ to $V_{p'}$ is of order $D^{1-\frac{d_{\ell+1}}{d_1}}$, which is of better decay than necessary.
This error can therefore be absorbed into the bound, as seen by comparing the equations
\[
X_{\hat{p}} = \{f_1(\tilde{x}) -p_1 = \cdots = f_\ell(\tilde{x}) - p_\ell = f_{\ell+1}(\tilde{x}) - z_0^{1 - \frac{d_{\ell+1}}{d_1}}p_{\ell+1} = \cdots f_k(\tilde{x})-z_0^{1 - \frac{d_k}{d_1}}p_k = 0\} 
\]
to the variety to which we want to compare,
\[
X_{p'}=\{f_1(x) - p_1 = \cdots = f_\ell(x)-p_\ell = f_{\ell+1}(x) = \cdots = f_k(x) = 0\}.
\]
Since this value has negative decay of $\rho^{-a}$ for $a = \frac{d_{\ell+1}}{d_1}$, $\delta$ can be chosen such that $-a <\delta - \frac{2}{d_1}$, or $\frac{2-d_{\ell+1}}{d_1} < \delta$.
Since $d_{\ell + 1} > 2$, $\delta$ can even be chosen to be negative.
\end{proof}

\section{Weighted H\"older Spaces}
\label{3Holder}
We construct weighted H\"older spaces so that we are able to localize the geometry to model manifolds of the form $X_t$ where we can invert the Laplacian.
Following Degeratu-Mazzeo \cite{Degeratu_2017}, we define a function $w$ that measures how close we are to the \textit{cone point}:
\[
w = \begin{cases}
1, & R > 2\kappa\rho\\
R/(\kappa \rho), & R \in (\kappa^{-1}\rho^{\frac{1}{d_1}},\kappa\rho)\\
\kappa^{-2}\rho^{\frac{1}{d_1}-1}, & R< \frac{1}{2}\kappa^{-1}\rho^{\frac{1}{d_1}}
\end{cases}
\]
for the same $\kappa$ as defined in Proposition \ref{acybound}.
We define the H\"older seminorm as
\[
[T]_{0,\gamma} = \sup_{\rho(z) > K}\rho(z)^\gamma w(z)^\gamma \sup_{z'\neq z, z \in B(z,c)}\frac{|T(z)-T(z')|}{|z-z'|^\gamma},
\]
where we compare using parallel transport when required.
We use this to define the H\"older weighted norm
\[
\|f\|_{C^{k,\alpha}_{\delta,\tau}} = \|f\|_{C^{k,\alpha}(\rho<2P))}  + [\rho^{k-\delta}\tau^{k-\tau}\nabla^kf]_{0,\alpha}+ \sum_{j=0}^k \sup_{\rho > P}\rho^{j-\delta}w^{j-\tau} |\nabla^j f|.
\]
This is re-expressible (up to a modification on a compact set making some regularized maximum of $1$ and $\rho$) as a normal H\"older norm with a conformal scaling of the metric by $\rho^{-2}w^{-2}$: $\|f\|_{C^{k,\alpha}_{\delta,\tau}} = \|\rho^{-\delta}w^{-\tau}\|_{C^{k,\alpha}_{ \rho^{-2}w^{-2}\omega}}$.
In this language, the decay of the Ricci potential in Proposition \ref{acybound} is $h \in C^{k,\alpha}_{\delta-2,-2}(V_t)$ for the value of $\delta$ chosen in the proposition.

We will invert the Laplacian only on a neighborhood of infinity, so we will have to extend functions outside a large disk to the interior.
The analysis of Proposition \ref{acybound} will give control over the metric in each of the regions such that a suitably scaled and smoothed reflection over the boundary of the disk will extend the function to the interior and only affect the norm by a bounded amount.
The proof follows the same techniques as in \Sz \cite[Proposition 8]{GaborPaper}.
\begin{proposition}[Extension of functions]\label{prop:extension}
Given $u \in C^{0,\alpha}_{\delta,\tau}(\rho^{-1}[A,\infty),\omega)$, there exists a bounded linear extension operator $E : C^{0,\alpha}_{\delta,\tau}(\rho^{-1}[A,\infty),\omega) \to C^{0,\alpha}_{\delta,\tau}(\cX)$ such that $Eu\vert_{\rho^{-1}[A,\infty)} = u$, and the norm of $E$ is independent of $A$.
\end{proposition}
Given this result, it is natural to define the same $C^{k,\alpha}_{\delta,\tau}(\rho^{-1}[A,\infty),\omega)$ norm over this subset of $\cX$ as the infimum of the $C^{k,\alpha}_{\delta,\tau}(\cX,\omega)$ norm over all extensions.

\section{Comparison to the Model Spaces}
\label{4modelcomp}
In this section, we show that $X_0$ and $X_{p'}$ model $\cX$ outside a sufficiently large compact set where $\rho > A$ for large enough $A \gg 0$.
The error of approximating the metric $\omega$ on $\cX$ by the Calabi-Yau metrics on $X_0$ and $X_{p'}$ is arbitrarily small in the above weighted spaces, with sufficient overlap to allow for gluing.
\subsection{Comparison to \texorpdfstring{$X_0$}{X0}}
We define a projection map from a region on $\cX$ that has geometry modeled by $X_0$.
We define the region $\cU = \left\{\rho > A,\, R > \Lambda \rho^{\frac{1}{d_1} }\right\}\cap \cX$ for large constants $A,\,\Lambda > 0$, and a map $G : \cU \to X_0$ where $\cU$ and $X_0$ are considered subsets of $\bC^{n+k+1}$ and $G$ is the nearest point projection in the cone metric $\iddbar(|x|^2 + |z|^2)$ for variables $x_1,\ldots,x_{n+k}$ and $z$.
The projection is such that $G(x,z) = (x',z)$ where $x'$ is the nearest point in the cone $V_0$ under the cone metric $\iddbar(R^2)$.
Since every fiber is asymptotic to $V_0$ and we are sufficiently far away with $\rho > A$ for $A \gg 0$, this is well-defined.

\begin{proposition}\label{prop:x0comp}
Given $\varepsilon > 0$, there exist constants $\Lambda,A \gg 0$ sufficiently large such that 
\[
|\nabla^i (G^*g_{X_0}-g)|_{g} < \varepsilon w^{-i}\rho^{-i},\quad i \leq k+1
\]
for $k$ in the $C^{k,\alpha}_{\delta,\tau}$ space.
Succinctly stated, $\|\nabla^i (G^*g_{X_0}-g)\|_{C^{k,\alpha}_{0,0}} < \varepsilon$ for any $\varepsilon > 0$.
\end{proposition}
\begin{proof}
This represents Regions \RNum{1}-\RNum{4}, and for the first three regions, the same computations from Proposition \ref{acybound} prove this as well, with the exception that this computation must be done using the Riemannian structure instead of the holomorphic structure since $G$ is smooth, but not holomorphic.
We need to show a bound $|D^{-2}\nabla^i (G^*g_{X_0}-g)|_{D^{-2}g} < \varepsilon$, and the same rescaling map gives us that the error induced is of size $D^{1-d_1}$.

For Region \RNum{4}, the computation is different than in Section \ref{2ACYmetric}, since before it was modeled on $X_{p'}$, but here it must be compared to $X_0$ instead.
For the estimate in this region, we will use the Conlon-Hein estimate \ref{CHbound} with an appropriate rescaling.
In this region, let $\rho \in (D/2,2D)$, $R \in (K/2,2K)$ and $\Lambda \rho^{\frac{1}{d_1}} < K < \frac{1}{2}\rho^\alpha$.
Pick some $z_0$ such that $|z_0| \in (D/4,4D)$ sufficiently close to $z$.
We introduce the same change of coordinates to $\tilde{x},\tilde{z},$ and $\tilde{r}$.
The equations of $\cX$ are given in these coordinates by
\[
\cX = \left\{f_i(\tilde{x}) = K^{-d_i}(K\tilde{z} + z_0)p_i = 0\right\} \text{ compared to }
X_0 = \left\{f_i(\tilde{x}) = 0\right\}.
\]
Since $|\tilde{z}| < 1$, the error introduced by ignoring the non-constant term with $\tilde{z}$ is $K^{1-d_1}$.
For $D \gg 0$ sufficiently large, this can be chosen to be within any $\varepsilon/2$, and therefore can be ignored.
Here $E$ is bounded by 
\[
E = \iddbar(K^{-2}|K\tilde{z} + z_0|^{\frac{2}{d_1}}\Phi((K\tilde{z} + z_0)^{-\frac{1}{d_1}}K \cdot \tilde{x}) - \tilde{r}^2),
\]
and as before, the estimate states that $(KD^{-\frac{1}{d_1}})^{-2-c} < C \Lambda^{-2-c}$.
The constant $\Lambda \gg 0$ can be chosen such that this is less than $\varepsilon/2$, finishing the proof.
\end{proof}
\subsection{Comparison to \texorpdfstring{$X_{p'}$}{Xp'}}
In this section, we study the region where the model variety is $X_{p'}$.
The region is specified by $R < \Lambda \rho^{\frac{1}{d_1}}$ and $\rho > A$.
Fix a $z_0$ with $|z_0| > A$, and define the region $\cV \subset \cX$ such that the point $(x,z) \in \cV$ if $|z-z_0| < B|z_0|^{\frac{1}{d_1}}$ for a large fixed constant $B$ and $R < \Lambda \rho^{\frac{1}{d_1}}$.
Define a new coordinate system as  
\[
\hat{x} = z_0^{-\frac{1}{d_1}}\cdot x,\quad \hat{x} = z_0^{-\frac{1}{d_1}}(z-z_0),\quad \hat{R} = |z_0|^{-\frac{1}{d_1}}R,
\]
and since $R$ can be small in this region, we need to define an auxiliary variable $\hat{\zeta} = \max (1,\hat{R})$.
There are bounds $|\hat{z}| < B, |\hat{R}| < C\Lambda$ for some fixed $C$ since $\rho \sim |z_0|$.
In these coordinates, the expression of $\cX$ is given by 
\[
\cX = \left\{f_i(\hat{x}) = z_0^{\frac{1-d_i}{d_1}}\hat{z}p_i + z_0^{1-\frac{d_i}{d_1}}p_i\right\}.
\]
We define a map $H :\cV \to X_{p'}$ by letting $H(\hat{x},\hat{z}) = (\hat{x}',\hat{z})$ where $\hat{x}'$ is simply the nearest point projection to 
\[
X_{p'} = \left\{f_1(\hat{x}) - p_1 = \cdots =f_{\ell}(\hat{x}) - p_{\ell} = f_{\ell+1}(\hat{x}) = \cdots f_k(\hat{x}) = 0\right\} = \bC \times V_{p'},
\]
with $\ell$ as before the maximal index such that $d_1 = d_\ell$.

\begin{proposition}
\label{prop:Xp'comp}
For any $\varepsilon > 0$, there exist constants $\Lambda, A > 0$ as functions of $B,\varepsilon, \Lambda$ such that $\| |z_0|^{\frac{2}{d_1}}H^*g_{X_{p'}} - g\|_{C^{k,\alpha}_{0,0}} < \varepsilon$.
Specifically, $|\nabla^i (H^*g_{X_{p'}} - |z_0|^{-\frac{2}{d_1}}g)|_{|z_0|^{-\frac{2}{d_1}}g} < \varepsilon \hat{\zeta}^{-i}$ for $i \leq k + 1$.
\end{proposition}
\begin{proof}
This follows almost directly according to the analysis in Section \ref{2ACYmetric}.
In Region \RNum{4} where $\rho \in (D/2,2D)$, $R \in (K/2,2K)$ and $\kappa^{-1}D^{\frac{1}{d_1}} < K < \Lambda D^{\frac{1}{d_1}}$, $\hat{\zeta} \sim \hat{R}$ and $|z_0| \sim D$.
We introduce new coordinates
\[
\tilde{x} = K^{-1}z_0^{\frac{1}{d_1}}\cdot \hat{z},\quad \tilde{z} = K^{-1}z_0^{\frac{1}{d_1}}\hat{z}.
\]
Let $\tilde{p} = (K^{-d_1}z_0p_1,\ldots, K^{-d_k}z_0p_k)$ which will be the comparison fiber. 
In the above coordinates, the expression of $\cX$ is given as
\[
\cX = \left\{f_i(\tilde{x}) = K^{-d_i}(K\tilde{z} + z_0)p_i = 0\right\} \text{ compared to }
X_{\tilde{p}} = \left\{f_i(\tilde{x}) = K^{-d_i} z_0p_i = 0\right\}
\]
under the projection operator map.
Again, we need to do the same computations, but with respect to the Riemannian metric, as the map $H$ is not holomorphic.
We use estimate \ref{CHbound} as in the proof of Region \RNum{4} in Theorem \ref{acybound}.
Using the same notation, there is the bound
\begin{align*}
    |\nabla^i (H^*g_{X_{p'}} - |z_0|^{-\frac{2}{d_1}}g)|_{|z_0|^{-{2}/{d_1}}g} &< K^{1-d_1-i}\tilde{z} + (|z_0|^{-\frac{1}{d_1}}K)^{-2-c}KD^{-1} \\
    &<K^{1-d_1-i}BK^{-1}z_0^{\frac{1}{d_1}} + (|z_0|^{-\frac{1}{d_1}}K)^{-2-c}KD^{-1} \\
    &< CBD^{\frac{1}{d_1}-1}
\end{align*}
where the first term in the second line is $CK^{-d_1}$, which is bounded by a constant times $D^{-1}$ and the second term is a decaying term containing $B$ times $KD^{-1} \sim D^{\frac{1}{d_1}-1}$ as desired.

Region \RNum{5} is modeled by $X_{p'}$.
This computation is identical to the above in Theorem \ref{acybound}.
We use the same change of variables applied to the $\hat{x},\hat{z}$ variables. 
The difference between $\hat{\zeta}$ and $\hat{R}$ can be ignored since it is a bounded prefactor to the decaying term.
This will reduce to examining 
\[
E = |z_0|^{-\frac{2}{d_1}}|z_0^{\frac{1}{d_1}}\tilde{z} + z_0|^{\frac{2}{d_1}}\Phi(z_0^{\frac{1}{d_1}}(z_0^{\frac{1}{d_1}}\tilde{z} + z_0)^{-\frac{1}{d_1}}\cdot\tilde{x}) - \Phi(\tilde{x}),
\]
and as before, the Taylor expansion gives the estimate $z_0^{\frac{1}{d_1}}(z_0^{\frac{1}{d_1}}\tilde{z} + z_0)^{-\frac{1}{d_1}} = 1 + O(D^{\frac{1}{d_1}-1})$, giving decay at the rate of $BD^{\frac{1}{d_1} - 1}$.
The other parts are identical as well, since $z_0^{-1/{d_1}}\cdot z_0 p\to p'$ as $|z_0| \to \infty$ by construction of $p'$, completing the proof.
\end{proof}
The results above can be used to show that the tangent cone at infinity of $\cX$ is $X_0$.
The proof of this follows the same structure as in \Sz \cite[Propositions 9, 10]{GaborPaper} as summarized in the following.
\begin{proposition}\label{prop:GHDist}
For any $\varepsilon > 0$, there is some $D \gg 0$ sufficiently large such that the Gromov-Hausdorff distance between the annular regions $\rho \in (D/2,2D)$ in $\cX$ with $\omega$ and $X_0$ with $\omega_{V_0} + \iddbar |z|^2$ is less than $D\varepsilon$.

Furthermore, for $D\gg 0$ sufficiently large, two points $x,x' \in \cX$ such that $d(o,x) = d(o,x') = D$ can be constructed as the endpoints of a curve $\gamma$ whose image is contained in the annulus $B(o,CD) \setminus B(o,C^{-1}D)$, and $\mr{length}_\omega(\gamma) < CD$ for some uniform constant $C$.
\end{proposition}
\section{Inverting the Laplacian and perturbing to a Calabi-Yau metric}\label{sec:5}
To complete the proof of Theorem \ref{thm:mainThm}, we use linear theory, together with the machinery of Tian-Yau \cite{YauTianII} and Hein \cite{HeinPhd}, to perturb the asymptotically Calabi-Yau metric to a globally Ricci-flat metric using the Calabi-Yau structure on the model spaces and their approximations of $\cX$.  
As these techniques are well-documented, we only outline the argument and refer the reader to \cite{HeinPhd, GaborPaper, YauTianII} for full details.

In the proper weighted spaces on $X_0$ and $X_{p'}$, the Laplacian is invertible.
The singularities are tractable using results on edge metrics from \cite{Degeratu_2017} and \cite{Rafe1991EllipticTO}.
The two main results on the invertibility of the Laplacian follow similarly to \Sz \cite[Propositions 23, 24]{GaborPaper} are as follows.
\begin{proposition}
\label{cor:lapX0}
On $X_0$, the operator $\Delta_{X_0}$ is invertible for the proper weights
\[
\Delta_{X_0} : C^{k,\alpha}_{\delta,\tau}(X_0) \to C^{k-2,\alpha}_{\delta-2,\tau-2}(X_0)
\]
if $\tau \in (2-m,0)$ and $\delta$ is generic.
On $X_{p'}$, the Laplacian operator $\Delta_{X_{p'}} $ is invertible for the proper weights 
\[
\Delta_{X_{p'}} :C^{k,\alpha}_\tau(X_{p'}) \to C^{k-2,\alpha}_{\tau-2}(X_{p'})
\]
if $\tau \in (2-m,0)$.
\end{proposition}
We now can perturb the asymptotically Calabi-Yau metric $\omega$ to a metric that is Ricci-flat outside of some compact set.
With the results of Section \ref{4modelcomp} demonstrating that the model spaces of $X_0$ and $X_{p'}$ are sufficient approximations in the various regions, and the above propositions inverting the Laplacians on the model space, we can now localize the manifold to the model spaces and invert the Laplacian on $\cX$ outside a sufficiently large compact set.
The proofs follow using the techniques laid out in \Sz \cite[Proposition 22]{GaborPaper}.
\begin{proposition}\label{prop:DelInvertNearInf}
Let $\tau \in (2-2n,0)$ and $\delta$ generic, for $A > 0$ sufficiently large.
Then the Laplacian 
\[
\Delta : C^{2,\alpha}_{\delta,\tau}(\rho^{-1}[A,\infty),\omega) \to C^{0,\alpha}_{\delta - 2,
\tau - 2}(\rho^{-1}[A,\infty),\omega)
\]
is surjective with inverse bounded independently of $A$.
\end{proposition}
From this, we can now perturb the asymptotically Calabi-Yau metric $\omega$ to be Ricci-flat outside this compact set, as in \Sz \cite[Proposition 25]{GaborPaper}.
\begin{proposition}
Let $A \gg 0$ be chosen sufficiently large, $\tau < 0$ sufficiently close to $0$, and $\delta < \frac{2}{d_1}$, (the conditions stated in Proposition \ref{acybound}).
There exists a function $u \in C^{2,\alpha}_{\delta,\tau}(\cX)$ with sufficiently small norm such that 
\[
(\omega + \iddbar u)^{n+1} = (\sqrt{-1})^{(n+1)^2}\Omega \wedge \ol{\Omega}
\]
on the set $\rho^{-1}[A,\infty)$.
\end{proposition}
\begin{proof}
The Ricci potential $\log [(\omega + \iddbar u)^{n+1}/[(\sqrt{-1})^{(n+1)^2}\Omega \wedge \ol{\Omega})]$ quantifies the failure of the guess $u$ to give rise to a Ricci-flat metric; it is 0 exactly for a Ricci-flat metric.
To converge at a Ricci-flat metric, the initial metric can be iteratively improved until it reaches the desired solution.
The conditions on $\tau$ and $\delta$ are such that $\omega$ and $\omega + \iddbar u$ are uniformly equivalent metrics for $u$ sufficiently small.
We define the set of viable $u$ as follows:
\begin{equation}\label{eqn:viableU}
\cB = \{u \in C^{2,\alpha}_{\delta,\tau} : \|u\|_{C^{2,\alpha}_{\delta,\tau} }\leq \epsilon_0\},
\end{equation}
where we chose $\epsilon_0$ such that every $u \in \cB$ is such that $\omega + \iddbar u$ is uniformly equivalent to $\omega$.
We define an operator $F$ \edit{which quantifies the failure of Ricci-flatness and vanishes at the desired Calabi-Yau metric.
This operator $F$ is defined as}
\[
F : \cB \to C^{0,\alpha}_{\delta-2,\tau-2}(\rho^{-1}[A,\infty)),\qquad F : u \mapsto \log \frac{(\omega + \iddbar u)^{n+1}}{(\sqrt{-1})^{(n+1)^2}\Omega \wedge \ol{\Omega}},
\]
where we restrict to the Ricci potential on $\rho^{-1}[A,\infty)$.
If we find $u$ such that $F(u) = 0$, we are done.

We can expand $F(u) = F(0) + \Delta_\omega u + Q(u)$ for some non-linear operator $Q$.
Using $P$ as the right inverse for $\Delta$ as described in the previous section, sufficiently far away from the origin, we must produce some function $u$ solving the PDE $u = P(-F(0) - Q(u))$.
For notation, we name this operator $N(u) = P(-F(0) - Q(u))$, and $u$ must be a fixed point of $N$.
We compute $Q(u) = F(u) - F(0) - \Delta_\omega u$, so we can take the difference
\[
Q(u) - Q(v) = F(u) - F(v) + \Delta_\omega v - \Delta_\omega u = F(u) - F(v) + \Delta_\omega (v-u),
\]
and by the chosen region, $u$ and $v$ are small \edit{in the $C^{2,\alpha}_{\delta,\tau}$ norm, so $\iddbar u$ is correspondingly small in $C^{0,\alpha}_{\delta-2,\tau - 2}$.
Therefore, the magnitude of $F(u) - F(v)$ is small measured in $C^{0,\alpha}_{\delta-2,\tau-2}$ for $\|u-v\|_{C^{2,\alpha}_{\delta,\tau}}$ small.}
There is an estimate on the difference: 
\[
\|Q(u) - Q(v)\|_{C^{0,\alpha}_{\delta-2,\tau-2}} \leq C(\|u\|_{C^{2,\alpha}_{2,2}} + \|v\|_{C^{2,\alpha}_{2,2}})\|u - v\|_{C^{2,\alpha}_{\delta,\tau}}
\]
as seen by expanding this operator as
\[
(\omega + \iddbar u)^{n+1} = \omega^{n+1} + (n+1) \iddbar u \wedge \omega^{n} + \binom{n+1}{2}\iddbar u \wedge \iddbar u \wedge \omega^{n-1} + \cdots.
\]
This can be computed as
\begin{align*}
    Q(u)-Q(v)&= F(u) - F(v) - \Delta_\omega\edit{(u-v)} \\
    &= \|\iddbar (u-v) \wedge \iddbar (u +v)\|_{C^{2,\alpha}_{2,2}}\\
    &\leq C(\|u\|_{C^{2,\alpha}_{2,2}} + \|v\|_{C^{2,\alpha}_{2,2}})\|u - v\|_{C^{2,\alpha}_{\delta,\tau}}.
\end{align*}

Therefore, there exists some constant $\epsilon_1$ such that the map $N(u) = P(-F(0) - Q(u))$ is a contraction on $\cB$ for $\|u\|_{C^{2,\alpha}_{2,2}} < \epsilon_1$.
To compare the norms $C^{2,\alpha}_{\delta,\tau}$ to $C^{2,\alpha}_{2,2}$, we use the relationship of the weight function
\[
\rho^\delta w^\tau \leq C\rho^{\delta - 2 + (\tau - 2)(\frac{1}{d_1}-1)}\rho^2w^2,
\]
which tells us that the $C^{2,\alpha}_{\delta,\tau}$ bound controls the $C^{2,\alpha}_{2,2}$ bound: 
\[
\|u\|_{C^{2,\alpha}_{2,2}} \leq C \|u\|_{C^{2,\alpha}_{\delta,\tau}}.
\]
Applying this bound on viable $\edit{u \in \cB}$ \edit{defined in equation \ref{eqn:viableU}}, the norm of $u-v$ controls the norm of $N(u) - N(v)$:
\[
\|N(u) - N(v)\| \leq \frac{1}{2}\|u-v\|,
\]
and this completes the proof, once verified that $N$ actually has its image contained in $\cB$.

By Proposition \ref{acybound}, $F(0) \in C^{0,\alpha}_{\delta' - 2,\tau-2}(\cX)$ for any $\delta' < \delta$ sufficiently close.
This gives the estimate
\[
\|F(0)\|_{C^{0,\alpha}_{\delta-2,\tau-2}(\rho^{-1}[A,\infty))} < CA^{\delta' - \delta},
\]
which is arbitrarily small for $A \gg 0$ large enough.
For $u \in \cB$, we can compute
\begin{align*}
    \|N(u)\|_{C^{2,\alpha}_{\delta,\tau}} &\leq \|N(0)\|_{C^{2,\alpha}_{\delta,\tau}} + \|N(u) - N(0)\|_{C^{2,\alpha}_{\delta,\tau}} \\
    &\leq C\|F(0\|_{C^{0,\alpha}_{\delta,\tau}(\rho^{-1}[A,\infty))} + \frac{1}{2}\|u\|_{C^{2,\alpha}_{\delta,\tau}}\\
    &= CA^{\delta'-\delta} + \frac{\epsilon_0}{2},
\end{align*}
so for $A$ large enough, this can be made to be less than $\epsilon_0$, and thus in $\cB$ as desired.
That $N$ is a contraction means it has a fixed point, which by definition has Ricci-potential 0. 
The fixed point metric is a Calabi-Yau metric on $\rho^{-1}[A,\infty)$.
\end{proof}

Applying Hein's PhD thesis \cite[Proposition 4.1]{HeinPhd} to the metric $\tilde{\omega} = \omega + \iddbar u$ for $u$ the solution of $F(u) = 0$ will perturb this metric to a global Ricci-flat metric, completing the proof of Theorem \ref{thm:mainThm}.
There are two conditions on $\cX$ necessary to apply this result: the $\mr{SOB}(2n+2)$ condition and the existence of a $C^{3,\alpha}$ quasi-atlas.

\edit{First, we verify that $\cX$ has the $\mr{SOB}(2n+2)$ property} from Hein \cite[Definition 3.1]{HeinPhd}, a condition on the \edit{(i) connectedness of annuli, (ii) a bound on the volume growth rate, and (iii) a lower bound on the Ricci curvature decaying quadratically.
The last condition (iii) on the Ricci curvature lower bound is immediately satisfied since the metric $\tilde{\omega}$ is Ricci-flat outside a large compact set.} 
\edit{For the connectedness condition (i), it is sufficient to verify the the condition of relatively connected annuli (RCA) \cite{Degeratu_2017}}, which can be shown using the Gromov-Hausdorff estimates to $X_0$. 
We utilize Proposition \ref{prop:GHDist} \edit{to verify the RCA condition}, which states that the tangent cone at infinity of $(\cX,\tilde{\omega})$ is $(X_0,\omega_{V_0} + \iddbar |z|^2)$, which is a metric cone.
Using the second half of Proposition \ref{prop:GHDist} and the fact that $\omega$ and $\tilde{\omega}$ are uniformly equivalent, this verifies the RCA property.
\edit{The above steps, following Degeratu-Mazzeo \cite{Degeratu_2017}, reduce the $\mr{SOB}(2n+2)$ condition to} showing that the asymptotic growth rate of balls is $r^{2(n+1)}$, \edit{verifying condition (ii)}.
We must produce some $C > 0$ such that for $r > C$, the volume of $B(x,r)$ satisfies
\begin{equation}\label{eqn:SOB}
C^{-1}r^{2(n+1)} < \mr{Vol}(B(x,r)) < Cr^{2(n+1)}
\end{equation}
for all $x \in \cX$ measured with respect to $\tilde{\omega}$.
Colding's volume convergence theorem \cite{Colding} applied to Gromov-Hausdorff limits will verify this.
Since $\tilde{\omega}$ is Ricci-flat outside a compact set, and has tangent cone at infinity with Euclidean volume growth, applying the volume convergence theorem of \cite{Colding} proves inequality \ref{eqn:SOB}.

\edit{Second, we must construct a $C^{3,\alpha}$ quasi-atlas}: charts covering $\cX$ of uniform size controlled in the $C^{3,\alpha}$ norm.
This \edit{construction follows from Propositions \ref{prop:x0comp} and \ref{prop:Xp'comp}}, which state that we can create charts of radius bounded by $C\rho w$ around any point bounded in the $C^{k,\alpha}$ norm.
The quantity $\rho w$ is unbounded as $\rho \to \infty$, as it is estimated as $\rho w > \kappa^{-2}\rho^{\frac{1}{d_1}}$.
\edit{The charts are uniformly bounded in size with respect to $\omega$}.
When we perturb the solution to $\tilde{\omega}$, the previous propositions only give us control in $C^{2,\alpha}$ initially.
\edit{Since $\tilde{\omega}$ solves the complex Monge-Amp\`ere equation, given the estimates from Section \ref{2ACYmetric}, we can apply elliptic regularity to boost the control to $C^{k,\alpha}$ as required.}
Therefore, this construction meets the requirements of Hein to apply \cite[Proposition 4.1]{HeinPhd}, which \edit{perturbs $\tilde{\omega}$ to a Calabi-Yau metric} on all of $\cX$ with the same tangent cone at infinity of $X_0$, proving Theorem \ref{thm:mainThm}.
\bibliographystyle{amsplain} 
\bibliography{biblio}

\providecommand{\bysame}{\leavevmode\hbox to3em{\hrulefill}\thinspace}
\providecommand{\MR}{\relax\ifhmode\unskip\space\fi MR }
\providecommand{\MRhref}[2]{%
  \href{http://www.ams.org/mathscinet-getitem?mr=#1}{#2}
}
\providecommand{\href}[2]{#2}
\begin{thebibliography}{10}

\bibitem{CheegerJeff1997Otso}
Jeff Cheeger and Tobias~H. Colding, \emph{On the structure of spaces with {R}icci curvature bounded below. {I}}, Journal of Differential Geometry \textbf{46} (1997), no.~3, 406--480 (eng).

\bibitem{YauChenPaper}
Shiu-Yuen Cheng and Shing-Tung Yau, \emph{On the existence of a complete {K}\"ahler metric on non-compact complex manifolds and the regularity of {F}efferman's equation}, Communications on Pure and Applied Mathematics \textbf{33} (1980), no.~4, 507--544.

\bibitem{chiu2022nonuniqueness}
Shih-Kai Chiu, \emph{Nonuniqueness of calabi-yau metrics with maximal volume growth}, 2022.

\bibitem{Colding}
Tobias~H. Colding, \emph{{R}icci curvature and volume convergence}, Annals of Mathematics \textbf{145} (1997), no.~3, 477--501.

\bibitem{ColdingUniqueTanCone}
Tobias~H. Colding and William~P. Minicozzi, II, \emph{On uniqueness of tangent cones for {E}instein manifolds}, Inventiones Mathematicae \textbf{196} (2014), no.~3, 515--588 (eng).

\bibitem{CollinsTristan2019SmaK}
Tristan~C. Collins and Gábor Székelyhidi, \emph{{S}asaki–{E}instein metrics and {K}–stability}, Geometry \& Topology \textbf{23} (2019), no.~3, 1339--1413 (eng).

\bibitem{Conlon_2013}
Ronan~J. Conlon and Hans-Joachim Hein, \emph{Asymptotically conical {C}alabi-{Y}au manifolds, {I}}, Duke Mathematical Journal \textbf{162} (2013), no.~15.

\bibitem{ConlonRonanJ2015AcCm}
Ronan~J. Conlon and Hans-Joachim Hein, \emph{Asymptotically conical {C}alabi–{Y}au metrics on quasi-projective varieties}, Geometric and Functional Analysis \textbf{25} (2015), no.~2, 517--552 (eng).

\bibitem{conlon-rochon}
Ronan~J. Conlon and Frédéric Rochon, \emph{New examples of complete {C}alabi-{Y}au metrics on $\mathbb{C}^n$ for $n\ge 3$}, Annales Scientifiques de l'\'{E}cole Normale Sup\'{e}rieure \textbf{54} (2021), no.~2, 259--303.

\bibitem{conlon2023warped}
Ronan~J. Conlon and Frédéric Rochon, \emph{Warped quasi-asymptotically conical calabi-yau metrics}, 2023.

\bibitem{Degeratu_2017}
Anda Degeratu and Rafe Mazzeo, \emph{{F}redholm theory for elliptic operators on quasi-asymptotically conical spaces}, Proceedings of the London Mathematical Society \textbf{116} (2017), no.~5, 1112--1160.

\bibitem{Demailly}
Jean-Pierre Demailly, \emph{Complex analytic and differential geometry}, Universit\'e de Grenoble I Institut Fourier, UMR 5582 du CNRS, 2012.

\bibitem{DonaldsonSun}
Simon Donaldson and Song Sun, \emph{{G}romov–{H}ausdorff limits of {K}\"ahler manifolds and algebraic geometry, {II}}, Journal of Differential Geometry \textbf{107} (2017), no.~2 (eng).

\bibitem{Gauntlett_2007}
Jerome~P. Gauntlett, Dario Martelli, James Sparks, and Shing-Tung Yau, \emph{Obstructions to the existence of {S}asaki-{E}instein metrics}, Communications in Mathematical Physics \textbf{273} (2007), no.~3, 803--827.

\bibitem{HeWeiyong2016FcaS}
Weiyong He and Song Sun, \emph{{F}rankel conjecture and {S}asaki geometry}, Advances in Mathematics (New York. 1965) \textbf{291} (2016), 912--960 (eng).

\bibitem{HeinPhd}
Hans-Joachim Hein, \emph{On gravitational instantons}, PhD Thesis, Princeton (2010).

\bibitem{LiCYC3}
Yang Li, \emph{A new complete {C}alabi–{Y}au metric on $\mathbb{C}^3$}, Inventiones Mathematicae \textbf{217} (2019), no.~1, 1--34 (eng).

\bibitem{Rafe1991EllipticTO}
Rafe Mazzeo, \emph{Elliptic theory of differential edge operators {I}}, Communications in Partial Differential Equations \textbf{16} (1991), 1615--1664.

\bibitem{GaborPaper}
G\'abor Sz\'ekelyhidi, \emph{Degenerations of $\mathbf{C}^n$ and {C}alabi-{Y}au metrics}, Duke Mathematical Journal \textbf{168} (2019), no.~14.

\bibitem{YauTian1}
Gang Tian and Shing-Tung Yau, \emph{Complete {K}\"ahler manifolds with zero {R}icci curvature. {I}}, Journal of the American Mathematical Society \textbf{3} (1990), no.~3, 579--609 (eng).

\bibitem{YauTianII}
\bysame, \emph{Complete {K}\"ahler manifolds with zero {R}icci curvature {II}}, Inventiones Mathematicae \textbf{106} (1991), no.~1, 27--60 (en).

\bibitem{Yau}
Shing-Tung Yau, \emph{On the {R}icci curvature of a compact {K}\"{a}hler manifold and the complex {M}onge-{A}mp\`ere equation. {I}}, Comm. Pure Appl. Math. \textbf{31} (1978), no.~3, 339--411. \MR{480350}

\end{thebibliography}
\end{document}